\documentclass[letterpaper, 10 pt, conference]{ieeeconf}

\IEEEoverridecommandlockouts                              

\usepackage{cite}
\usepackage{amsmath,amssymb,amsfonts}
\usepackage{algorithmic}
\usepackage{microtype}
\usepackage{graphicx}
\usepackage{algorithm,algorithmic}
\usepackage{hyperref}
\hypersetup{hidelinks=true}
\usepackage{textcomp}
\usepackage{mathtools}
\usepackage{amsthm}
\usepackage{xcolor}
\usepackage{tabularx}
\theoremstyle{plain} 
\usepackage[table]{xcolor}
\usepackage{tikz}
\usetikzlibrary{arrows.meta,positioning}

\newcommand{\Q}{\mathbb{Q}}
\newcommand{\R}{\mathbb{R}}

\newcommand{\ip}[2]{\left\langle #1,#2\right\rangle}
\newcounter{prob}

\newcommand{\mcl}[1]{\mathcal{#1}}

\newcommand{\mbf}[1]{\mathbf{#1}}

\newcommand{\bmat}[1]{\begin{bmatrix} #1 \end{bmatrix}}
\renewcommand{\ip}[2]{\left\langle #1,#2\right\rangle}

\newcommand{\norm}[1]{\left\lVert{#1}\right\rVert}
 
\newcommand{\err}{\ensuremath{err}} 
\newcommand{\PI}{\pmb{\Pi}}
\newcommand{\pie}{\scalebox{0.9}[1.2]{$\mathit{\Pi}$}}

\newcommand{\fourpi}[4]{\hspace{0.5mm}\pie\hspace{-1.mm}\left[\begin{array}{c|c}
#1&#2\\\hline #3 & \left\{#4\right\}
\end{array}\right]}

\newtheorem{thm}{Theorem}

\newtheorem{defn}[thm]{Definition}
\newtheorem{lem}[thm]{Lemma}
\newtheorem{cor}[thm]{Corollary}

\newtheorem{prop}[thm]{Proposition}

\let\bl\bigl
\let\bbl\Bigl

\let\br\bigr
\let\bbr\Bigr

\title{\LARGE \bf Dual Representations and Proxy Characterization of $H_2$ Norm for Partial Integral Equations}

\author{Danilo Braghini$^{1}$, Tristan Thomas$^{2}$, Sachin Shivakumar$^{3}$, Javad M. Velni$^{2}$, and Matthew M. Peet$^{1}$
\thanks{This work was supported by the National Science Foundation under grants No. 2337751 and 2429973.}
\thanks{$^{1}$ Danilo Braghini\{{\tt\small dbraghini@asu.edu}\} and Matthew M. Peet are with the School for Engineering of Matter, Transport and Energy at Arizona State University, Tempe, AZ, USA. $^{2}$ Tristan Thomas and Javad Velni \{{\tt\small tthoma7@clemson.edu, javadm@clemson.edu}\} are with Department of Mechanical Engineering at Clemson University. $^{3}$Sachin Shivakumar\{{\tt\small sshivakumar@lanl.gov}\} is with CNLS and T-5 Division of Los Alamos National Laboratory.
        }%
}
\allowdisplaybreaks

\begin{document}
\maketitle
\thispagestyle{empty}
\pagestyle{plain}

\begin{abstract}
Duality and dual representations are essential tools for systems theory and controller synthesis. For any finite-dimensional causal linear time-invariant (CLTI) system, there exists a dual CLTI system with identical $H_\infty$ and $H_2$ norm. These relations are found by simply taking the adjoint of the transfer function. For nonlinear, time-varying or infinite-dimensional systems, however, such a transfer function may not exist or be difficult to obtain. For duality in the $H_\infty$ norm, this obstacle is overcome by using the equivalent time-domain, induced $L_2$-norm characterization. For $H_2$ norms there is no such equivalent characterization and hence proxy norms such as induced energy to peak and impulse to energy norms are used. These proxy norms bound the $H_2$ norm in the CLTI case. However, neither of these $H_2$ proxy norms equal the same $H_2$ proxy of the canonical dual representation. In this paper, we overcome this by proposing a dual representation whose induced energy to peak proxy equals the impulse to energy proxy of the primal and vice-versa. This relation is developed in the context of Partial Integral Equation (PIE) representations. The result is then used to establish primal-dual operator inequalities for both proxy norms similar in form to the primal-dual LMIs used to characterize the $H_2$ norm.  
\end{abstract}

\section{Introduction}

Solution of the $H_2$ and $H_\infty$-optimal controller synthesis problems firmly established state-space, optimization, and Linear Matrix Inequalities (LMIs) as essential tools in optimal and robust control. An essential feature in both results, however, was frequency-domain characterization of system norms and transfer to time-domain via, e.g., the Kalman-Yakubovich-Popov (KYP) Lemma. This frequency domain interpretation then allowed for equivalent characterization of norms in terms of a dual system obtained by taking the adjoint of the transfer function, so that $\norm{C(sI-A)^{-1}B}=\norm{B^T(sI-A^T)^{-1}C^T}$ implies that systems $\{A,B,C\}$ and $\{A^T,C^T,B^T\}$ have equivalent $H_2$ and $H_\infty$ norm. When the KYP lemma or $H_2$ equivalent is applied to the dual system, the resulting matrix inequalities admit a convexification step through variable substitution.

When extending these LMI-based methods to $H_2$ and $H_\infty$ analysis and control of nonlinear, time-varying, uncertain, and infinite-dimensional systems, however, the transfer function interpretation fails. Specifically, 
 $H_2$ and $H_\infty$ norms are signal norms -- defined in terms of a transfer function. For systems where no such transfer function exists, the norm is undefined. In the $H_\infty$ case, this may be overcome by using a time-domain characterization as the induced $L_2$ gain. The induced system norm is equivalent to the $H_\infty$ norm for Causal Linear Time-Invariant (CLTI) systems and may be readily characterized using LMIs. Furthermore, for linear  infinite-dimensional systems which admit a Partial Integral Equation (PIE) representation, a dual PIE representation exists with equivalent induced $L_2$ gain. This result was used to design $H_\infty$-optimal state-feedback controllers in~\cite{shivakumar2026}. For the 
 $H_2$--norm, however, no analogous 
 characterization exists~\cite{stoorvogel1993}.

To address this problem, two time-domain \textit{proxies} for the $H_2$-norm have been proposed which bound the $H_2$ for CLTI ODE state-space systems. 
We consider the induced impulse-to-energy and induced energy-to-peak norms. For a system $\dot x=Ax+Bw, z=Cx$, the induced energy-to-peak norm, or generalized $H_2$--norm, is defined as
{
\setlength{\abovedisplayskip}{2pt}
\setlength{\belowdisplayskip}{2pt}
\setlength{\abovedisplayshortskip}{2pt}
\setlength{\belowdisplayshortskip}{2pt}
\[
\|G\|_{\infty,2}:=\sup_{T\geq 0} \sup_{\|w\|_{L_2[0,T]}=1} \norm{z}_{L_\infty[0,T]},
\]
}
and has been used for analysis and control of systems with polytopic uncertainty~\cite{palhares2000}, time-varying parameters~\cite{tan2002}, sampled-data~\cite{kim2017}, time-varying delay~\cite{gao2003}, nonlinearities~\cite{li2017}, and bounded uncertainties
~\cite{schwenkel2026}. 
Unlike the induced $L_2$-norm proxy for $H_\infty$, the induced energy-to-peak norm is asymmetric: the 
transform $\{A,B,C\}\rightarrow\{A^T,C^T,B^T\}$ does not preserve the norm. An alternative $H_2$ proxy is the induced impulse-to-energy norm, defined for the system 
$\dot x=Ax, z=Cx, x(0)=Bv$ as 
\[
\|G\|_{2,i}
:=\sup_{T\geq 0} \sup_{\|v\|_{2}=1} \norm{z}_{L_2[0,T]}.
\]

This proxy was used in~\cite{iwasaki1996} for systems with time-varying uncertainties. Again, the canonical duality transform $\{A,B,C\}\rightarrow\{A^T,C^T,B^T\}$ does not preserve the induced impulse-to-energy norm. The asymmetry poses a challenge to development of dual LMI or operator inequality representations of either norm.

In this paper, we address this problem by proposing a class of systems which encompasses both norms. To allow for generalization to infinite-dimensional systems, we pose this class in the Partial Integral Equation (PIE) representation~\cite{shivakumar2024,peet2021}. This class of systems takes the form $\mcl T\dot x=\mcl Ax+\mcl Bu$, $x(0)=\mcl Hv$, $z=\mcl Cx$ with associated ordered tuple $\{\mcl T, \mcl A,\mcl B,\mcl C,\mcl H\}$ which is a subset of the algebra of Partial Integral (PI) operators. We then define the primal system as the tuple $\{\mcl T, \mcl A,0,\mcl C,\mcl B\}$ and the dual system as $\{\mcl T^*, \mcl A^*,\mcl C^*,\mcl B^*,0\}$. The primal form is then admissible for the impulse-to-energy norm and the dual is admissible for the energy-to-peak norm. Our goal is to show that $\|G(\mcl T,\mcl A,0,\mcl C,\mcl B)\|_{2,i}
    =
    \|G(\mcl T^*,\mcl A^*,\mcl C^*,\mcl B^*,0)\|_{\infty,2}$ and 
$ \|G(\mcl T,\mcl A,\mcl B,\mcl C,0)\|_{\infty,2}
    =\|G(\mcl T^*,\mcl A^*,0,\mcl B^*,\mcl C^*)\|_{2,i}
$.

To show that the proxy of the primal is equivalent to the corresponding proxy of the dual, we first establish an intertwining relationship between solutions of the primal and dual systems -- showing that the inner product of the initial state of the dual and the final state of the primal equals the inner product of the output of the primal with the backward-in-time input of the dual (Lem.~\ref{lem:intertwining}). For CLTI state-space systems, such a relationship can be inferred by combining the results in~\cite{chellaboina2000,bobba2004} with~\cite{peregudin2024} and for semigroups a related involution property can be found in~\cite{tucsnak2009}. Lem.~\ref{lem:intertwining} is used to obtain the desired norm-duality in Thm.~\ref{thm:dualnorm}.

Having established the main result, we then examine the question of whether Thm.~\ref{thm:dualnorm} can be used to define convex conditions 
for controller synthesis. 
To accomplish this, we generalize known primal LMIs for both proxy norms~\cite{iwasaki1996,scherer1997} to operator inequalities. We use the duality result to establish dual versions of these operator inequalities. An operator Schur complement gives two dilated inequalities for each proxy, arranged as alternative primal and dual pairs -- mirroring the structure of the classical $H_2$ LMI formulations.

The results are implemented in PIETOOLS~\cite{manual2025}, which automates the conversion from PDE or delay system to PIE representation and uses semidefinite programming to solve linear operator inequalities. Although developed in the PIE framework, the results may find application in analysis and control of finite-dimensional systems -- 
descriptor or differential algebraic systems which have a similar structure. Finally, the results are tested 
on PDEs and time-delay systems.  


The paper is organized as follows. The problem is formulated in Sec.~\ref{sec:problem}, where the proxy norms are introduced. Sec.~\ref{sec:duality} establishes the intertwining identity and duality theorem. Secs.~\ref{sec:LPIs} and~\ref{sec:dilatedLPIs} present the primal, dual, and dilated operator inequalities. Numerical results are given in Sec.~\ref{sec:numerics} and show that for single-input single output systems with known $H_2$ norm, the bounds given by the operator inequalities are tight.

\noindent\textbf{Notation:} We will use 
uppercase calligraphic font $\mcl A$ for operators. Given Hilbert spaces $X$ and $Y$, $\mcl L(X,Y)$ denote the set of bounded linear operators from $X$ to $Y$. We denote the Hilbert adjoint by $\mcl A^* \in \mcl L(Y,X)$. 
 $L_{\infty}^{n}[\Omega]$, and $L_2^n[\Omega]$ are the spaces of $\R^n$--valued Lebesgue bounded, and Lebesgue square-integrable functions, respectively, on the domain $\Omega$, 
with the standard inner product. $W^{n}_d[\Omega]$ is the Sobolev space of $d$ times $L_2$--differentiable $\R^{n}$--valued functions on $\Omega$. \vspace{-2 mm}
 \section{System Norms and Problem Statement}\label{sec:problem}\vspace{-1 mm}
We consider the problem of computing optimal upper-bounds for proxies of the $H_2$--norm of a linear PDE or delay system expressed using the PIE state-space representation~\cite{shivakumar2024}. PIE representations are parameterized by Partial Integral (PI) operators, defined as follows. 

	\begin{defn}\label{def:4PI}
		Given a matrix $P$ and polynomials $Q_1,Q_2,R_0$, $R_1$, $R_2$, we say that $\mcl P$ is a \textbf{4-PI operator} if
			\begin{align*}\small
				&\left(\mcl P \bmat{x\\\mbf x} \right)(s) := \bmat{Px + \int_{a}^{b}Q_1(\theta)\mbf x(\theta)d\theta\\Q_2(s)x+ \left(\mcl R\mbf x\right) (s)}, \text{where}~\left(\mcl R\mbf x\right)(s)\\
	&\hspace{-0.5mm}=\hspace{-0.5mm}R_0(s) \mbf x(s) +\hspace{-1.5mm}\int_{a}^s  \hspace{-1.5mm}R_1(s,\theta)\mbf x(\theta)d \theta+\hspace{-1.5mm}\int_s^b \hspace{-1.5mm}R_2(s,\theta)\mbf x(\theta)d \theta.		
			\end{align*}
            
            We denote the map from $P,Q_i,R_i \mapsto \mcl P$ as \[\mcl P= \fourpi{P}{Q_1}{Q_2}{R_i}\subset \mcl L(\R^m \times L_2^{n},\R^p \times L_2^{q})\] and, $\mcl P \in \PI$ if $\mcl P$ is a 4-PI operator. 
	\end{defn}
	
	 If $p=m$ and $q=n$, $\PI$ is closed under composition, addition, and 
     adjoint, defining a *-algebra akin to matrices for finite-dimensional state-spaces. Explicit formulae for these operations can be found in~\cite{shivakumar2024}.

We now propose the following parameterized class of systems which will be used to define our $H_2$ proxy norms.

 \begin{defn}\label{def:system}
 Given $\mcl T,\mcl A,\mcl B,\mcl C,\mcl H \in \PI$ and $v$, $w$, we say $z=G(\mcl T,\mcl A,\mcl B,\mcl C,\mcl H)(w,v)$ if there exists $x(t)$ such that 
 \begin{align}
\!\partial_t( \mathcal T x(t)) \!= \!\mathcal A x(t)\!+\!\mcl Bw(t),
   \mathcal T x(0) \!=\! \mcl H v, 
z(t) = \mathcal Cx(t),
    \label{eq:PIE}
\end{align}
  \end{defn}

\noindent\textbf{Note:} The operators $\mcl T,\mcl A,\mcl B,\mcl C,\mcl H$ determine the dimensions of $v,w,z,x$. In this paper, however, we presume $v(t),w(t),z(t)$ are real vector-valued functions 
constrained to the domains of $\mcl H, \mcl B$ and range of $\mcl C$. 
%

We use the following notion of solution~\cite{shivakumar2024}.
\begin{defn}\label{def:piesol}
Given $\mcl{T,A,B, C, H}\in\PI$, 
$\{x,z\}$ is a \textbf{solution of $G(\mcl T,\mcl A,\mcl B,\mcl C,\mcl H)$} if 
$\mcl Tx$ is Fréchet differentiable, and 
$\{z, x$, $w$,$v\}$ satisfy Eqn.~\eqref{eq:PIE}.
\end{defn}
We will presume throughout that the PIE has been obtained from a well-posed PDE or time-delay system and hence is, itself, well-posed in the following sense. 
\begin{defn}\label{def:wellposed}
We say $G(\mcl T,\mcl A,\mcl B,\mcl C,\mcl H)$ is well-posed if, for every $w\in L_2$ and $v \in \R^{n_v}$, there exists a solution $\{x,z\}$ of  $G(\mcl T,\mcl A,\mcl B,\mcl C,\mcl H)$.
\end{defn}

Partial Integral Equations can be used to equivalently represent the state and input-output properties of a large class of PDEs. Specifically, given a PDE, explicit constructions of the operators $\mcl{T,A,B, C, H}$ can be found in~\cite{shivakumar2024}. Conditions for well-posedness may be found in~\cite{jagt2026}.

Since transfer functions (and the corresponding trace norm) are typically unavailable for PDEs/PIEs, formal statements of duality in $H_2$ require the use of proxy system norms, as defined in the following subsection.

\subsection{Impulse-to-Energy Norm of PDEs}
  
In the context of PIEs, induced impulse to energy norm is defined for partial systems of the form $z=G(\mcl T,\mcl A,0,\mcl C,\mcl B)v$, where recall by Defn.~\ref{def:system} this implies
\begin{align}
   \partial_t( \mathcal T x(t)) &= \mathcal A x(t),\;
    \mathcal T x(0)= \mathcal B v,\;
    z(t) = \mathcal Cx(t).
    \label{eq:primal_aux}
\end{align}
for some $x(t)$.

\begin{defn}[Induced impulse-to-energy norm]\label{def:2,i_norm}
Given $\mcl{T,A,B,C} \in \PI$ with $G(\mcl T,\mcl A,0,\mcl C,\mcl B)$ well-posed, define
    \begin{equation}
\|G(\mcl T,\mcl A,0,\mcl C,\mcl B)\|_{2,i}:=   \sup_{T\ge 0} \sup_{\substack{z,\mbf x\, \text{satisfy~\eqref{eq:primal_aux}}\\ \norm{v}_2=1}} \norm{z}_{L_2[0,T]}.
    \label{eq:primal_proxy_norm}
\end{equation}
\end{defn}
Given $z=G(\mcl T,\mcl A,0,\mcl C,\mcl B)v$ -- an autonomous system with initial state $\mcl B v$, $\|G(\mcl T,\mcl A,0,\mcl C,\mcl B)\|_{2,i}\le \gamma$ implies that $\norm{z}_{L_2}\le \gamma \norm{v}$. That is, an impulse $v$, acting through $\mcl B$ produces an initial condition $\mcl Bv$ and the energy of the output is bounded by the size of the initial condition.

\subsection{Energy-to-Peak Norm of PDEs}

In the context of PIEs, energy-to-peak norm is defined for partial systems of the form $z=G(\mcl T,\mcl A,\mcl B,\mcl C,0)v$, where recall by Defn.~\ref{def:system} this implies
    \begin{align}
    \mcl T\dot{x}(t)&=\mcl Ax(t)+\mcl B v(t),\;
    \mcl Tx(0)=0,\; 
    z(t)=\mathcal C x(t).
    \label{eq:dualPIE}
\end{align}

\begin{defn}[Induced energy-to-peak norm]\label{def:inf,2_norm}
    Given $\mcl{T,A,B,C} \in \PI$ with $G(\mcl T,\mcl A,\mcl B,\mcl C,0)$ well-posed, define
    \begin{equation}
\|G(\mcl T,\mcl A,\mcl B,\mcl C,0)\|_{\infty,2}=\sup_{T >0}
   \sup_{\substack{v,x\, \text{satisfy~\eqref{eq:dualPIE}}\\\norm{v}_{L_2[0,T]}=1} } \|z (T)\|.
    \label{eq:dual_norm}
\end{equation}
\end{defn}

Given $z=G(\mcl T,\mcl A,\mcl B,\mcl C,0)v$ -- a system with input $v$ and zero initial state, $\|G(\mcl T,\mcl A,\mcl B,\mcl C,0)\|_{\infty,2}\le \gamma$ implies that $\norm{z}_{L_\infty} \le \gamma \norm{v}_{L_2}$. That is, the finite-energy input $v$ produces a bounded output $z$, and the peak value of the output is bounded by the energy of the input.

\subsection{$\|G\|_{2,i}$ and $\|G\|_{\infty,2}$ are proxies for $\norm{G}_{H_2}$ }

In the context of ODEs, induced impulse-to-energy and energy to peak norms  have been used as a proxy for $H_2$ norm for systems where transfer functions need not exists. However, when transfer functions exist and $H_2$ norm is finite, the two proxies are equivalent to (but not equal to) the $H_2$ norm. This can be expressed in the partial PIE system construction as the following extension of well-known results for ODEs~\cite{peregudin2024}.
\begin{prop}\label{prop:H2equi_PIE}
Given $\mcl{T,A,B, C}\in\PI$, where the domain of $\mcl B$ is $\R^{n_w}$, and the image of $\mcl C$ is $\R^{n_z}$, suppose $\hat G(s)=\mcl C(s\mcl T-~\mcl A)^{-1}\mcl B$ exists. 
Then we have the following.
    \[ \hspace{-2mm}
\|G(\mcl T,\mcl A,0,\mcl C,\mcl B)\|_{2,i} \le\hspace{-1mm} \|\hat G\|_{H_2}\hspace{-1mm}\le \sqrt{n_w} \|G(\mcl T,\mcl A,0,\mcl C,\mcl B)\|_{2,i}
\]
  \[\hspace{-3mm}
\|G(\mcl T,\mcl A,\mcl B,\mcl C,0)\|_{\infty,2} \le\hspace{-1mm} \|\hat G\|_{H_2}\hspace{-1mm}\le \sqrt{n_z} \|G(\mcl T,\mcl A,\mcl B,\mcl C,0)\|_{\infty,2}
\]
\end{prop}
\subsection{Problem Formulation}

The goal of the paper is to relate the impulse-to-energy norm of a PDE system to the energy-to-peak norm of its dual. Specifically, given an input-output PIE representation of the system, as per Defn.~\ref{def:system}, and the norms of Defn.~\ref{def:2,i_norm} and Defn.~\ref{def:inf,2_norm}, we want to prove the following symmetries.
\[
\|G(\mcl T,\mcl A,0,\mcl C,\mcl B)\|_{2,i}
    =
    \|G(\mcl T^*,\mcl A^*,\mcl C^*,\mcl B^*,0)\|_{\infty,2},
\]
\[
\|G(\mcl T,\mcl A,\mcl B,\mcl C,0)\|_{\infty,2}
    =\|G(\mcl T^*,\mcl A^*,0,\mcl B^*,\mcl C^*)\|_{2,i}.
\]
This result is established in Thm.~\ref{thm:dualnorm} and included in Table~\ref{tab:input-output-norms} along with previously established primal-dual relations for $H_{\infty}$~\cite{shivakumar2026} and impulse-to-peak norm~\cite{thomas2026}.
%
\begin{table*}
\centering
\caption{Relations between induced system norms for primal and dual systems, classified by input and output spaces. The diagonal boxes indicate ``self-duality'' relations where norm is preserved between input/output spaces and primal/dual representations -- citations are included for these results.}
\label{tab:input-output-norms}
\renewcommand{\arraystretch}{1.15}
\setlength{\tabcolsep}{8pt}
\begin{tabular}{|c|c|c|c}
\hline
\textbf{Output $\backslash$ Input}
&
\textbf{Energy ($L_2$)}
&
\textbf{Impulse (Initial Condition)}
\\
\hline\rule{0pt}{2.2em}
\textbf{Energy ($L_2$)}
&
\shortstack{Energy-to-energy (see \cite{shivakumar2026})\\
$\|G(\mcl T,\mcl A,\mcl B,\mcl C,0)\|_{H_{\infty}}= \|G(\mcl{T^*,A^*,C^*,B^*},0)\|_{H_\infty}$}
&
\shortstack{Impulse-to-energy (Thm.~\ref{thm:dualnorm})\\
$\|G(\mcl T, \mcl A,0,
\mcl C,\mcl B)\|_{2,i}= \|G(\mcl T^*,\mcl A^*,\mcl C^*,\mcl B^*,0)\|_{\infty,2}$}
\\
\hline
\rule{0pt}{2.2em}
\textbf{Peak ($L_\infty$)}
&
\shortstack{Energy-to-peak (Thm.~\ref{thm:dualnorm})\\
$\|G(\mcl T, \mcl A,\mcl B,\mcl C,0)\|_{\infty,2}=\|G(\mcl T^*,\mcl A^*,0,\mcl B^*,\mcl C^*)\|_{2,i}$} 
&
\shortstack{Impulse-to-peak (see \cite{thomas2026})\\
$\|G(\mcl T, \mcl A,0,\mcl C,\mcl B)\|_{\infty,i}=\|G(\mcl{T^*,A^*},0,\mcl{B^*,C^*})\|_{\infty,i}$}
\\
\hline
\end{tabular}
\end{table*}

Based on this duality result, we propose linear PI operator inequality (LPI) conditions for both the impulse-to-energy norm (Thm.~\ref{thm:LPI_2,i}) and the energy to peak norm (Thm.~\ref{thm:LPI_inf,2}). These results are combined with a Schur complement in Hilbert spaces to obtain both primal and dual LPIs for the impulse-to-energy norm in Thm.~\ref{thm:LPIs_2,i_dilated} and for the energy-to-peak norm in Thm.~\ref{thm:LPIs_inf,2_dilated}. The resulting inequalities can potentially be used for optimal control and estimator synthesis, and generalize existing LMIs for ODEs.

\section{Duality in $\norm{\cdot}_{e,i}$ and $\norm{\cdot}_{\infty,2}$}\label{sec:duality}


In this section, we establish the main duality result. This is based on the following intertwining property showing that the inner product of the initial state of the dual system $G(\mcl T^*,\mcl A^*,0,\mcl B^*,\mcl C^*)$ and the final state of the primal system ($G(\mcl T, \mcl A,\mcl B,\mcl C,0)$) equals the inner product of the output of the primal with the backward-in-time input of the dual. 
\begin{lem}\label{lem:intertwining}
    Given $\mcl{T,A,B,C} \in \PI$, suppose $x,v$ satisfy Eqn.~\eqref{eq:primal_aux} for $\{\mcl T,\mcl A,\mcl B,\mcl C\}$, and $\bar x,\bar v$ satisfy Eqn.~\eqref{eq:dualPIE} for $\{\mcl T^*, \mcl A^*,\mcl C^*,\mcl B^*\}$. Then, for any $T>0$, we have the following
\begin{align}
    \label{eq:intertwining2}
 \bl(\mcl B^* \bar x(T) \br)^T v=  \int_0^T\ip{\bar v(T-t)}{\mcl C x(t)}dt.
\end{align}
\end{lem}
\begin{proof}
    Suppose $x,v$ satisfy Eqn.~\eqref{eq:primal_aux}, and $\bar v,\bar x$ satisfy Eqn.~\eqref{eq:dualPIE}. Define $V(t)
=\left\langle
\bar x(t),\mathcal T x(T-t)
\right\rangle$ for $t \in [0,T]$. Then
\begin{align*}
\dot V(t)&=\ip{\mathcal T^*\dot{\bar x}(t)}{x(T-t)}
-
\ip{\bar x(t)}{\mathcal T\dot x(T-t)}
\\
&=
\ip{
\mathcal A^*\bar x(t)
+\mathcal C^*\bar v(t)}{x(T-t)}-
\ip{\bar x(t)}{\mathcal A x(T-t)}\\
&=
\ip{\bar v(t)}{ \mcl Cx(T-t) }.
\end{align*}
Hence, integrating in $t$ over $[0,T]$ we have 
\begin{align*}
&V(T)-V(0)=\int_0^T\ip{\bar v(t)}{ \mcl Cx(T-t) }dt\\
&=-\int_T^0\ip{\bar v(T-s)}{ \mcl Cx(s) }ds=\int_0^T\ip{\bar v(T-s)}{ \mcl Cx(s) }ds
\end{align*}
where we used the change of variables $s=T-t$. Now from Eqn.~\eqref{eq:primal_aux}, we have $\mcl T x(0)= \mathcal B v$ and from Eqn.~\eqref{eq:dualPIE}, we have $\mcl T^*\bar x(0)=0$. Thus
\begin{align*}
&\bl(\mcl B^* \bar x(T) \br)^T v=\ip{\bar x(T)}{\mcl B v}=\ip{\bar x(T)}{\mcl T x(0)}\\&
=\ip{\bar x(T)}{\mcl T x(0)}-\ip{\mcl T^*\bar x(0)}{x(T)}\\
&=\ip{\bar x(T)}{\mcl T x(0)}-\ip{\bar x(0)}{\mcl Tx(T)}=V(T)-V(0)\\
&=\int_0^T\ip{\bar v(T-t)}{\mcl C x(t)}dt
\end{align*}
as desired.
\end{proof}
This intertwining property is use to obtain the following main result.
\begin{thm}\label{thm:dualnorm}
Given $\mcl{T,A,B,C}\in \PI$, let $G(\mcl T,\mcl A,0,\mcl C,\mcl B)$ and $G(\mcl T^*,\mcl A^*,\mcl C^*,\mcl B^*,0)$ be well-posed. Then we have that
\[
    \|G(\mcl T,\mcl A,0,\mcl C,\mcl B)\|_{2,i}
    =
    \|G(\mcl T^*,\mcl A^*,\mcl C^*,\mcl B^*,0)\|_{\infty,2}.
\]
\end{thm}

\begin{proof}

Recall that   
\[
\|G(\mcl T^*,\mcl A^*,\mcl C^*,\mcl B^*,0)\|_{\infty,2}=    \sup_{T>0}
   \sup_{\substack{v,x\, \text{satisfy~\eqref{eq:dualPIE}}\\\norm{v}_{L_2[0,T]}=1} } \|\mcl B^* x(T)\|,
\]
\[
\text{and }\;\|G(\mcl T,\mcl A,0,\mcl C,\mcl B)\|_{2,i}:= \sup_{T>0}   \sup_{\substack{v,z,\mbf x\text{satisfy~\eqref{eq:primal_aux}}\\ \norm{v}_2=1}} \norm{z}_{L_2[0,T)}.
\]


To show that $\|G(\mcl T^*,\mcl A^*,\mcl C^*,\mcl B^*,0)\|_{\infty,2}\le \|G(\mcl T,\mcl A,0,\mcl C,\mcl B)\|_{2,i}$, we establish that for any $T>0$ and any $\bar{v}$, $\bar{x}$ which satisfy satisfy Eq. (5) with $\norm{\bar{v}}_{L_2[0,T]}=1$, we have that $|\mcl B^* x(T)|\le \|G(\mcl T,\mcl A,0,\mcl C,\mcl B)\|_{2,i}$.

Specifically, for $T>0$ and any $\bar{v}$, $\bar{x}$ which satisfy Eq. (5) with $\norm{\bar{v}}_{L_2[0,T]}=1$, let $v = \frac{1}{\norm{\mcl B^*\bar{x}(T)}}\mcl B^*\bar{x}(T)$ and $x,z$ satisfy Eq. (2) for this $v$. Then from Lem.~\ref{lem:intertwining},
\begin{align*}
    \norm{\mcl B^*\bar{x}(T)} &= (\mcl B^*\bar{x}(T))^Tv= \int_0^T\ip{v(T-t)}{\mcl Cx(t)}dt\\
    &\hspace{-1.5cm}\le\! \norm{v}_{L_2[0,T]}\!\norm{z}_{L_2[0,T]}\! =\!\norm{z}_{L_2[0,T]}\!\le\! \|G(\mcl T,\mcl A,0,\mcl C,\mcl B)\|_{2,i}
\end{align*}
as desired, where the final inequality holds from the definition of $\|G(\mcl T,\mcl A,0,\mcl C,\mcl B)\|_{2,i}$ and the fact that $v,x,z$ satisfy Eq. (2) with $\norm{v}=1$. Since the bound holds uniformly for any $T>0$ and any $\bar{v}$, $\bar{x}$ which satisfy Eq. (5) with $\norm{\bar{v}}_{L_2[0,T]}=1$, we conclude that
\begin{align*}
\|G(\mcl T^*,\mcl A^*,\mcl C^*,\mcl B^*,0)\|_{\infty,2}&=    \sup_{T>0}
   \sup_{\substack{v,x\, \text{satisfy~\eqref{eq:dualPIE}}\\\norm{v}_{L_2[0,T]}=1} } |\mcl B^* x(T)|\\
   &\le \|G(\mcl T,\mcl A,0,\mcl C,\mcl B)\|_{2,i}.
\end{align*}

Now, to show that $\|G(\mcl T,\mcl A,0, \mcl C,\mcl B)\|_{2,i}\le \|G(\mcl T^*,\mcl A^*,\mcl C^*,\mcl B^*,0)\|_{\infty,2}$, we similarly establish that for any $T>0$ and any $v$, $z$ $x$ which satisfy Eq.~\eqref{eq:primal_aux} with $\norm{v}=1$, we have that $\norm{z}_{L_2[0,T)} \le \|G(\mcl T^*,\mcl A^*,\mcl C^*,\mcl B^*,0)\|_{\infty,2}$.

Specifically, for $T>0$ and any $v$, $z$ $x$ which satisfy Eq.~\eqref{eq:primal_aux} with $\norm{v}=1$, let $\bar v(t) = \frac{1}{\norm{z}_{L_2[0,T)}}z(T-t)$ and $\bar x$ satisfy Eq.~\eqref{eq:dualPIE} for this $\bar v$. Then from Lem.~\ref{lem:intertwining},
\begin{align*}
\norm{z}_{L_2[0,T]}&=\frac{1}{\norm{z}_{L_2[0,T)}}\int_0^T \ip{z(t)}{z(t)}dt\\
&\hspace{-1.4cm}=\int_0^T \ip{\bar v(T-t)}{\mcl C x (t)}dt=\bl(\mcl B^* \bar x(T) \br)^T v\\
&\hspace{-1.4cm}\le \norm{\mcl B^* \bar x(T) }\!\norm{ v}\!=\!\norm{\mcl B^* \bar x(T) }\!\le\! \|G(\mcl T^*,\mcl A^*,\mcl C^*,\mcl B^*,0)\|_{\infty,2}
\end{align*}
where the final inequality holds from the definition of $\|G(\mcl T^*,\mcl A^*,\mcl C^*,\mcl B^*,0)\|_{\infty,2}$ and the fact that $\bar v,\bar x$ satisfy Eq.~\eqref{eq:dualPIE} with $\norm{\bar v}=1$. Since the bound holds uniformly for any $T>0$ and any  $v$, $z$ $x$ which satisfy Eq.~\eqref{eq:primal_aux} with $\norm{v}=1$, we conclude that
\begin{align*}
\|G(\mcl T,\mcl A,0,\mcl C,\mcl B)\|_{2,i}&=    \sup_{T>0}   \sup_{\substack{v,z,\mbf x\text{satisfy~\eqref{eq:primal_aux}}\\ \norm{v}_2=1}} \norm{z}_{L_2[0,T)}\\
&\le \|G(\mcl T^*,\mcl A^*,\mcl C^*,\mcl B^*,0)\|_{\infty,2}
\end{align*}
as desired. 
We conclude that $\|G(\mcl T,\mcl A,0,\mcl C,\mcl B)\|_{2,i}=\|G(\mcl T^*,\mcl A^*,\mcl C^*,\mcl B^*,0)\|_{\infty,2}$.
\end{proof}

We are now in position to address the question of how to obtain optimal bounds on $\|G(\mcl T,\mcl A,0,\mcl C,\mcl B)\|_{2,i}$ and $\|G(\mcl T,\mcl A,\mcl B,\mcl C,0)\|_{\infty,2}$.

\section{Operator Inequalities for $\norm{G}_{2,i}$ and $\norm{G}_{\infty,2}$}\label{sec:LPIs}
Having proven the main result in Thm.~\ref{thm:dualnorm}, we will apply this result to obtain both primal and dual forms of linear operator inequality (LPI) which upper bound the $\norm{\cdot}_\mbf{2,i}$ (Thm.~\ref{thm:LPI_2,i}, Sec.~\ref{thm:LPIs_2,i_dilated}) and $\norm{\cdot}_\mbf{\infty,2}$ (Thm.~\ref{thm:LPI_inf,2}, Sec.~\ref{thm:LPIs_2,i_dilated}). Before proceeding, however, we first use Lyapunov arguments to obtain primal undilated operator inequalities for  $\norm{G}_{2,i}$ and $\norm{G}_{\infty,2}$ --  generalizing results obtained in~\cite{iwasaki1996} and~\cite{scherer1997}, respectively.\vspace{-1 mm}

\subsection{An LPI for Impulse-to-Energy Norm}\label{sec:e2p}\vspace{-1 mm}

In this Subsection, we provide an LPI which characterizes $\norm{\cdot}_\mbf{2,i}$, based on a generalization of Thm.~1 in~\cite{iwasaki1996}.
%


\begin{thm}\label{thm:LPI_2,i}
    Given $\mcl{T,A,B,C} \in \PI$, suppose there exist $\mcl P=\mcl P^* \in \PI$, and $\gamma, \epsilon>0$, such that $\mcl P \succcurlyeq \epsilon I$, satisfying \begin{align}
\mathcal A^*\mathcal P\mathcal T
+\mathcal T^*\mathcal P\mathcal A
+
\mathcal C^*\mathcal C
&\preccurlyeq 0,\;
\mathcal B^*\mathcal P\mathcal B
\preccurlyeq \gamma I.
\label{eq:lpi1b}
\end{align}
Then, $\norm{G(\mcl T,\mcl A,0,\mcl C,\mcl B)}_{2,i}^2 \leq \gamma$. 
\end{thm}
\begin{proof}
We prove that if the inequalities hold for some $\mcl P\succcurlyeq \epsilon I, \gamma>0$, then $\norm{G(\mcl T,\mcl A,0,\mcl C,\mcl B)}_{2,i}^2 \leq \gamma$, recalling that
\[
\|G(\mcl T,\mcl A,0,\mcl C,\mcl B)\|_{2,i}:= \sup_{T>0}   \sup_{\substack{v,z,\mbf x\text{satisfy~\eqref{eq:primal_aux}}\\ \norm{v}_2=1}} \norm{z}_{L_2[0,T)}.
\]
Suppose there exist $\mcl P\succcurlyeq \epsilon I, \gamma>0$ satisfying 
Eqn.~\eqref{eq:lpi1b}. 
For any $T>0$, let $z, x, v$ satisfy Eqn.~\eqref{eq:primal_aux} for $t\in[0,T]$, given some initial condition with $\norm{v}_2=1$. Define the storage functional $V(x(t))=\ip{ x(t)}{\mcl{T^*PT}x(t)}_{L_2[a,b]}$. Since $\mcl P$ is coercive, $V(x)\ge \epsilon \norm{\mcl Tx}^2$ for any $x\in L_2$. Then, taking time derivative along solutions of Eqn.~\eqref{eq:primal_aux}, we have that
    \begin{align*}
        \dot V(x(t))&=\ip{ x(t)}{(\mcl{A^*PT+T^*PA}), x(t)}_{L_2[a,b]},
    \end{align*}
    and, since $\mcl P$ satisfies Eqn.~\eqref{eq:lpi1b}, we further have that
    \[
    \dot V( x(t)) \leq -\ip{x(t)}{\mcl{C^*C} x(t)}_{L_2[a,b]}.
    \]
    Integrating $\dot{V}$ w.r.t. $t$ over $[0,T]$, 
    we have that
    \begin{align*}
  \norm{z}_{L_2[0,T]}^2&=\int_0^T\ip{x(s)}{\mcl{C^*C}x(s)}ds \\
    &\leq V(x(0))- V(x(T))\leq V(x(0))\\
    &=\ip{x(0)}{\mcl{T^*PT}x(0)}=\ip{\mcl B v}{\mcl{PB} v}
    \end{align*}
Then, since $\mcl P, \gamma$ satisfy Eqn.~\eqref{eq:lpi1b}, we have that
   \[
    \norm{z}_{L_2[0,T]}^2 \leq \bl|\ip{\mcl B v}{ \mcl{PB} v}\br| \bbr|\leq  \gamma \norm{v}_2^2=\gamma.
   \]
    Since the inequality holds uniformly for any $T>0$ and solution of Eqn. \eqref{eq:primal_aux}, by the definition of the impulse-to-energy norm, we conclude that
$\norm{G(\mcl{T,A},0,C,B)}_{2,i}^2 \leq \gamma$, as desired.
\end{proof}

\subsection{An LPI for the Energy-to-Peak Norm}\label{sec:e2p}
In this Subsection, we provide an operator inequality which characterizes $\norm{\cdot}_\mbf{\infty,2}$, based on a generalization of the result in Sec.~III- D of~\cite{scherer1997}.

\begin{thm}\label{thm:LPI_inf,2}
 Given $\mcl{T,A,B,C} \in \PI$, suppose there exist $\mcl P=\mcl P^* \in \PI$, and $\gamma, \epsilon>0$, such that $\mcl P \succcurlyeq \epsilon I$, satisfying
\begin{align}
         \bmat{-\gamma I & \mcl{B^*PT}\\ \mcl{T^* P B} & \mcl{A^*PT+T^*PA}} &\preccurlyeq0, \;
        \mcl{C^*C}-\mcl{T^*PT}
        \preccurlyeq0.\label{eq:lpi1d}
     \end{align}
Then, $\norm{G(\mcl{T,A,B,C},0)}_{\infty,2}^2 \leq \gamma$
\end{thm}
\begin{proof}
 We prove that if there exist $\mcl P \succcurlyeq \epsilon I, \gamma >0$ satisfying 
 Eqn.~\eqref{eq:lpi1d}, then $\|G(\mcl T,\mcl A,\mcl B,\mcl C,0)\|_{\infty,2}^2 \leq \gamma$, recalling that
\[\|G(\mcl T,\mcl A,\mcl B,\mcl C,0)\|_{\infty,2}=    \sup_{T >0}
   \sup_{\substack{v,x\, \text{satisfy~\eqref{eq:dualPIE}}\\\norm{v}_{L_2[0,T]}=1} } \|z (T)\|.\]

     Suppose there exist $\mcl P\succcurlyeq \epsilon I, \gamma>0$ satisfying 
     Eqn.~\eqref{eq:lpi1d}. For any $T>0$, suppose $z,x, v$ satisfy Eqn.~\eqref{eq:dualPIE}, given some input with $\norm{v}_{L_2[0,T]} = 1$. 
     Define the storage functional \(V(x(t))=\ip{x(t)}{ \mcl{T^* PT} x(t)}.\) Then, $V(x)\ge \epsilon \norm{\mcl T\bar{x}}^2$ and 
   {\small
    \begin{align*}
        \dot V(x(t))=\ip{x(t)}{(\mcl{A^* PT+T^*PA}) x(t)}+2\ip{v}{\mcl{B^*PT} x(t)},
    \end{align*}}
implying that, since $\mcl P$ and $\gamma$ satisfy Eqn.~\eqref{eq:lpi1d},
    \begin{align*}
            \dot V(x(t))&-\gamma \norm{v(t)}_2^2=\\
            &\hspace{-4mm}\ip{\bmat{v(t)\\ x(t)}}{ \bmat{-\gamma I & \mcl{B^*PT}\\\mcl{T^*PB}& \mcl{A^*PT+T^*PA}}\bmat{v(t)\\ x(t)}} \leq 0.
    \end{align*}
Integrating $\dot{V}$ w.r.t. $t$ over $[0,T]$, we get
\begin{align*}
V(x(T))-V(x(0))-\gamma \norm{v}_{L_2[0,T]}^2 \leq 0.
\end{align*}
Then, since $\mcl P$ and $\gamma$ satisfy Eqn.~\eqref{eq:lpi1d}, we have that
\begin{align*}
          \norm{z(T)}_2^2
          &=
          \ip{x(T)}{ \mcl{C^*C}x(T)}\leq\ip{x(T)}{\mcl{T^*PT} x(T)}\\&= V(x(T))
          =V(x(T)-V(x(0))\leq\gamma\norm{v}^2=\gamma
    \end{align*}

    Since the inequality holds uniformly for any $T>0$ and solutions of Eqn. \eqref{eq:dualPIE}, 
by the definition of the energy-to-peak norm, we conclude that $
\norm{G(\mcl{T,A,B,C},0)}_{\infty,2}^2 \leq \gamma$, as desired.  
\end{proof}

\vspace{- 1mm}

\section{Primal/Dual LPIs for $\|G\|_{2,i}$, and $\|G\|_{\infty,2}$}\label{sec:dilatedLPIs}\vspace{-1 mm}

The LPIs in Thms.~\ref{thm:LPI_2,i} and~\ref{thm:LPI_inf,2} are undilated and in primal form (i.e. expressed using $\mcl A^* \mcl P$). The convexification step in controller synthesis requires a dilated dual form. To obtain the dual form in each case, we combine the primal LPI with the duality of norms from Thm.~\ref{thm:dualnorm}. We then dilate the resulting LPIs using the following generalization of the Schur complement Lemma~\cite{braghini2025b}.

\begin{lem}\label{lem:schur} Given Hilbert spaces $X$ and $Y$, $\mcl R\in\mcl L(X)$, $\mcl S\in\mcl L(Y)$, and $\mcl U\in\mcl L( Y, X)$ with $\mcl S \succcurlyeq \epsilon I$ for some $\epsilon>0$. Then the following are equivalent.
\begin{enumerate} 
\item 
$\displaystyle \bmat{\mcl S&\mcl U\\\mcl U^*&\mcl R}\succcurlyeq 0 $.
\item $\displaystyle \mcl R-\mcl U^* \mcl S^{-1} \mcl U \succcurlyeq 0$.
\end{enumerate} 
\end{lem} 

\subsection{Primal and Dual LPIs for $\|G\|_{2,i}$}
The first undilated result is an application of Thm. \ref{thm:dualnorm} to Thm.~\ref{thm:LPI_inf,2}.

\begin{cor}\label{cor:LPI_2,i_dual}
 Given $\mcl{T,A,B,C} \in \PI$, suppose there exist $\mcl P=\mcl P^* \in \PI$, and $\gamma, \epsilon>0$, such that $\mcl P \succcurlyeq \epsilon I$, satisfying
\begin{align}
         \bmat{-\gamma I & \mcl{CPT^*}\\ \mcl{T P C^*} & \mcl{APT^*+TPA^*}} &\preccurlyeq0,\;
        \mcl{BB^*}-\mcl{TPT^*}
        \preccurlyeq0.\label{eq:lpi6b_}
     \end{align}

Then, $\norm{G(\mcl{T,A},0,\mcl C,\mcl B)}_{2,i}^2 \leq \gamma$
\end{cor}
\begin{proof}
    As before, we first prove that if 
    Eqn.~\eqref{eq:lpi6b_} hold for some $\mcl P\succcurlyeq \epsilon I, \gamma>0$, then $\norm{G(\mcl{T^*,A^*},\mcl C^*,\mcl B^*,0)}_{\infty,2}^2 \leq \gamma$, and use the duality result in Thm.~\ref{thm:dualnorm}. 
    Suppose there exist $\mcl P\succcurlyeq \epsilon I, \gamma>0$ satisfying 
    Eqn.~\eqref{eq:lpi6b_}.
    Now, let $\mcl{T'=T^*}$, $\mcl{A'=A^*}$, $\mcl{B'=C^*}$, and $\mcl{C'=B^*}$. Since $\mcl{T,A,B,C} \in \PI \subset \mcl L(X)$ and $(A^*)^*=A$ for any $A\in \mcl L(X)$, we can verify that the conditions of Thm.~\ref{thm:LPI_inf,2} hold for $\mcl T',\mcl A',\mcl B',\mcl C'$ since
    \begin{align*}
                 \bmat{-\gamma I & \mcl{(B')^*PT'}\\ \mcl{(T')^* P B'} & \mcl{(A')^*PT'+(T')^*PA'}} &\preccurlyeq0\;\\
        \mcl{(C')^*C'}-\mcl{(T')^*PT'}&
        \preccurlyeq 0,
    \end{align*}    
    implying \(\norm{G(\mcl T',\mcl A',\mcl B',\mcl C',0)}_{\infty,2}^2\leq \gamma.\)
Now, from Thm.~\ref{thm:dualnorm}, we can conclude that
\begin{align*}
    \gamma&\ge \norm{G(\mcl T'\!,\!\mcl A'\!,\!\mcl B'\!,\mcl C',0)}_{\infty,2}^2\\
    &=\norm{G(\mcl{(T')^*\!,\!(A')^*}\!,0,\!\mcl{(B')^*\!,\!(C')^*})}_{2,i}^2\\
    &=\norm{G(\mcl{(T^*)^*,(A^*)^*},0,\mcl{(C^*)^*,(B^*)^*})}_{2,i}^2\\
    &=\norm{G(\mcl{T,A},0,\mcl{C,B})}_{2,i}^2.
\end{align*}
\end{proof}\vspace{- 3mm}

We now dilate the LPIs from Thm.~\ref{thm:LPI_2,i} and Cor.~\ref{cor:LPI_2,i_dual} combined with the norm duality result in Thm.~\ref{thm:dualnorm} to obtain the following primal/dual characterization of $\|G\|_{2,i}$. 

\begin{thm}\label{thm:LPIs_2,i_dilated}
    Given $\mcl{T,A,B,C} \in \PI$, suppose at least one of the following 
    statements hold.
    \begin{enumerate}
        \item There exist $\mcl P=\mcl P^* \in \PI$, and $\gamma, \epsilon>0$, such that $\mcl P \succcurlyeq \epsilon I$,  satisfying, \begin{align}
\bmat{-\gamma I & \mcl C \\\mcl C^* &\mathcal A^*\mathcal P\mathcal T
+\mathcal T^*\mathcal P\mathcal A}
&\preccurlyeq 0, \,
\bmat{\mcl P &\mcl{PB}\\ \mcl{B^*P} &\gamma I } &\succcurlyeq 0
\label{eq:lpi2b}
\end{align}
\item  There exist $\mcl P=\mcl P^* \in \PI$, and $\gamma,\epsilon>0$, such that $\mcl P \succcurlyeq \epsilon I $, satisfying
\begin{align}
         \bmat{-\gamma I & \mcl{CPT^*}\\ \mcl{T P C^*} & \mcl{APT^*+TPA^*}} &\preccurlyeq0, \,
         \bmat{\gamma I & \mcl B^*\\\mcl B &  \mcl{TPT^*}}
         \succcurlyeq 0\label{eq:lpi3b}.
     \end{align}
    \end{enumerate}

Then, 
$\norm{G(\mcl T,\mcl A,0,\mcl C,\mcl B)}_{2,i}
\leq \gamma$.
\end{thm}
\begin{proof}


We prove each statement separately

\noindent \textbf{Proof of $1)$}:  Suppose $1)$ holds for some $\mcl P \succcurlyeq \epsilon I, \gamma>0$. 
Then, $\mcl P \in \PI \subset \mcl L(X)$, and $\mcl P,\gamma$ satisfy Eqn.~\eqref{eq:lpi2b}. Now, we apply Lem.~\ref{lem:schur} using $\mcl U=-\mcl C \in \mcl L(X)$, $\mcl R=-(\mcl{A^*PT+T^*PA}) \in \mcl L(X)$, and $\mcl S = \gamma I \in \mcl L(X)$. Since $\mcl S \succcurlyeq \epsilon I$ for any $\epsilon\in(0,\gamma]$, the conditions of Lem.~\ref{lem:schur} are satisfied. Then, we have that 
\[
\bmat{-\mcl S&-\mcl U\\-\mcl U^*&-\mcl R}= \bmat{-\gamma I  & \mcl C\\\mcl C^*& \mcl{A^*PT+T^*PA} }\preccurlyeq0,
\]
\[
 \text{and hence}\,\bmat{\mcl S&\mcl U\\\mcl U^*&\mcl R}\succcurlyeq0.
\]
Then, statement 1) of Lem.~\ref{lem:schur} holds. This implies
\[
-\bbl( \mcl{A^*PT+T^*PA+C^*C}\frac{1}{\gamma}\bbr)=\mcl{R-U^*}\mcl S^{-1}\mcl U \succcurlyeq 0.
\]

 Furthermore, $\mcl P,\gamma$ satisfy Eqn.~\eqref{eq:lpi2b}.
    Now, we apply Lem.~\ref{lem:schur} using $\mcl U=\mcl{PB} \in \mcl L(X)$, $\mcl R=\gamma I \in \mcl L(X)$, and $\mcl S = \mcl P \in \mcl L(X)$. Then, since $\mcl S =\mcl P \succcurlyeq \epsilon I$, the conditions of Lem.~\ref{lem:schur} are satisfied. Then, we have that
    \[
    \bmat{\mcl S&\mcl U\\\mcl U^*&\mcl R}= \bmat{\mcl P & \mcl{PB}\\\mcl{B^*P}&\gamma I}\succcurlyeq0,
    \]
    and hence, statement 1) of Lem.~\ref{lem:schur} holds. This implies
    \[
 \gamma I -\mcl{B^*PB} = \gamma I -\mcl{B^*PP}^{-1}\mcl{PB}=\mcl{R-U^*}\mcl S^{-1}\mcl U \succcurlyeq 0.
\]
  Next, let $\hat{\mcl P} =\mcl P\gamma$. Then, we have that $\hat{\mcl P} \succcurlyeq \gamma \epsilon I$ and
 \begin{align*}
        \mcl{A^*\hat{P}T+T^* \hat{P} A+C^*C}  &\preccurlyeq0,
        \; \gamma^2 I-\mcl{B^*}\hat{\mcl P}\mcl{B} \succcurlyeq 0, 
    \end{align*}
and hence $\mcl{\hat P}=\mcl P \gamma$, and $\gamma^2$ satisfy the conditions of Thm.~\ref{thm:LPI_2,i}. This implies that $\norm{G(\mcl T,\mcl A,0,\mcl C,\mcl B)}_{2,i}^2 \leq \gamma^2$,  and hence
$\norm{G(\mcl T,\mcl A,0,\mcl C,\mcl B)}_{2,i} \leq \gamma$, as desired. 

 \noindent \textbf{Proof of $2)$}: Next, we prove statement $2)$ by showing that if there exist $\mcl P \succcurlyeq \epsilon I, \gamma >0$ satisfying $2)$, then $\|G(\mcl T^*,\mcl A^*,\mcl C^*,\mcl B^*,0)\|_{\infty,2} \leq \gamma$, recalling that
\[\|G(\mcl T,\mcl A,\mcl B,\mcl C,0)\|_{\infty,2}=    \sup_{T >0}
   \sup_{\substack{v,x\, \text{satisfy~\eqref{eq:dualPIE}}\\\norm{v}_{L_2[0,T]}=1} } \|z (T)\|.\]
      Suppose there exist $\mcl P\succcurlyeq \epsilon I, \gamma>0$ satisfying 
      Eqn.~\eqref{eq:lpi3b}. 
        Now, we apply Lem.~\ref{lem:schur}, using $\mcl U =\mcl{B^*} \in \mcl L(X)$, $\mcl R=\mcl{TPT^*} \in \mcl L(X)$, and $\mcl S=\gamma I \in \mcl L(X)$. Since $\gamma I \succcurlyeq \epsilon I$ for any $\epsilon\in(0,\gamma]$, the conditions of Lem.~\ref{lem:schur} are satisfied. 
        We have 
\[
  \bmat{\mcl S&\mcl U\\\mcl U^*&\mcl R}=\bmat{\gamma I & \mcl B^*\\\mcl B & \mcl{TPT^*}}\succcurlyeq 0,
\]        
    and hence, statement 1) of Lem.~\ref{lem:schur} holds. This implies
    \[
     \mcl{TPT^*} -\gamma^{-1}\mcl{BB^*} =\mcl R- \mcl U^* \mcl S^{-1}\mcl U  \succcurlyeq 0.
    \]
 Next, let $\hat{\mcl P} =\mcl P\gamma$. Then, we have that $\hat{\mcl P} \succcurlyeq \gamma \epsilon I$, and
\begin{align*}
          \bmat{-\gamma^2 I & \mcl{C\hat PT^*}\\ \mcl{T \hat P C^*} & \mcl{A\hat PT^*+T\hat PA^*}} &\preccurlyeq0, \;\mcl{T\hat PT^*} -\mcl{BB^*}\succcurlyeq 0,
\end{align*}
and hence $\mcl{\hat P}=\mcl P \gamma$, and $\gamma^2$ satisfy the conditions of Cor.~\ref{cor:LPI_2,i_dual}. This implies $\norm{G(\mcl T,\mcl A,0,\mcl C,\mcl B)}_{2,i}^2\leq \gamma^2$, and hence $\norm{G(\mcl T,\mcl A,0,\mcl C,\mcl B)}_{2,i}\leq \gamma$ is proven.
\end{proof}

\textbf{Note:} Items $1)$ of Thm.~\ref{thm:LPIs_2,i_dilated} and Thm.~\ref{thm:LPI_2,i} can be shown to be equivalent. 
Equivalence between Thm.~\ref{thm:LPI_2,i} and statement $2)$ of Thm.~\ref{thm:LPIs_2,i_dilated}, however, is precluded by that fact that the map from $X$ to $D\subset X$, where $D$ is a Sobolev subspace, given by the PIE operator $\mcl T \in \mcl L(X)$ is not an isomorphism (even though it is bijective onto $D$, as shown in~\cite{peet2021aut}): $\mcl T^{-1}$ is not a bounded operator. If $\mcl T^{-1} \in \mcl L(X)$, then the change of variables $\mcl {\hat P}=(\mcl{T^*PT})^{-1}$ provides equivalence between statements $1)$ and $2)$ of Thm.~\ref{thm:LPIs_2,i_dilated}. Hence, in that restrictive case, $2)$ of Thm.~\ref{thm:LPIs_2,i_dilated} is also equivalent to Thm.~\ref{thm:LPI_2,i}.

\subsection{Primal and Dual LPIs for $\|G\|_{\infty,2}$}

In this subsection, we apply the duality result of Thm.~\ref{thm:dualnorm} to Thm.~\ref{thm:LPI_2,i}, generalizing the LMI Lem.~2.1 of~\cite{rotea1993}.
\begin{cor}\label{cor:LPI_inf,2_dual}
    Given $\mcl{T,A,B,C} \in \PI$, suppose there exist $\mcl P=\mcl P^* \in \PI$, and $\gamma,\epsilon>0$, such that $\mcl P \succcurlyeq \epsilon I$, satisfying\begin{align}
\mathcal A\mathcal P\mathcal T^*
+\mathcal T\mathcal P\mathcal A^*
+
\mathcal B\mathcal B^*
&\preccurlyeq 0, \;\; 
\mcl{CPC^*} 
\preccurlyeq \gamma I.
\label{eq:lpi4b}
\end{align}
Then, $\norm{G(\mcl{T,A,B,C},0)}_{\infty,2}^2 \leq \gamma$. 
\end{cor}
\begin{proof}
We prove that if 
Eqn.~\eqref{eq:lpi4b} hold for some $\mcl P\succcurlyeq \epsilon I, \gamma>0$, then $\norm{G(\mcl{T^*,A^*},0,\mcl B^*,\mcl C^*)}_{2,i}^2 \leq \gamma$. We use duality, in Thm.~\ref{thm:dualnorm}, to assert that $\norm{G(\mcl{T,A,B,C},0)}_{\infty,2}^2 \leq \gamma$. 
Suppose there exist $\mcl P\succcurlyeq \epsilon I, \gamma>0$ satisfying 
Eqn.~\eqref{eq:lpi4b}.
Let $\mcl{T'=T^*}$, $\mcl{A'=A^*}$, $\mcl{B'=C^*}$, and $\mcl{C'=B^*}$. By Thm.~11.3 of Chapter II of~\cite{gohbergVol1}, $(A^*)^*=A$ for any $A\in \mcl L(X)$. Since $\mcl{T,A,B,C} \in \PI \subset \mcl L(X)$, we have $(\mcl T')^*=\mcl T$, $(\mcl A')^*=\mcl A$, $(\mcl C')^*=\mcl B$, and $(\mcl B')^*=\mcl C$ and consequently, conditions of  Thm.~\ref{thm:LPI_2,i} hold,
\begin{align*}
    (\mathcal A')^*\mathcal P\mathcal T'
+(\mathcal T')^*\mathcal P\mathcal A'
+(\mathcal C')^*\mathcal C'
\preccurlyeq 0,\;
\mcl{(B')^*PB'} 
\preccurlyeq \gamma I,
\end{align*}
and hence we conclude that \(\norm{G(\mcl T',\mcl A',0,\mcl C',\mcl B')}_{2,i}^2\leq \gamma.\) 
Lastly, from Thm.~\ref{thm:dualnorm}, we have that
\begin{align*}
    \gamma&\ge \norm{G(\mcl T'\!,\!\mcl A'\!,0,\!\mcl C'\!,\!\mcl B')}_{2,i}^2\\
    &=\norm{G((\mcl T')^*\!,\!(\mcl A')^*\!,\!(\mcl C')^*\!,\!(\mcl B')^*\!,\!0)}_{\infty,2}^2\\&=\norm{G(\mcl{(T^*)^*,(A^*)^*,(B^*)^*,(C^*)^*},0)}_{\infty,2}^2\\
    &=\norm{G(\mcl{T,A,B,C},0)}_{\infty,2}^2.
\end{align*}
\end{proof}

We now dilate the LPIs from Thm.~\ref{thm:LPI_inf,2} and Cor.~\ref{cor:LPI_inf,2_dual} combined with the norm duality result in Thm.~\ref{thm:dualnorm} to obtain the following primal/dual characterization of $\|G\|_{\infty,2}$. 

\begin{thm}\label{thm:LPIs_inf,2_dilated}
    Given $\mcl{T,A,B,C} \in \PI$, suppose at least one of the following statements hold.
    \begin{enumerate}
        \item There exist $\mcl P=\mcl P^* \in \PI$, and $\gamma,\epsilon>0$, such that $\mcl P \succcurlyeq \epsilon I$, satisfying
\begin{equation}\label{eq:lpi5b}
\hspace{-6mm}\bmat{-\gamma I & \mcl B^* \\\mcl B&\mathcal A\mathcal P\mathcal T^*
+\mathcal T\mathcal P\mathcal A^*}\preccurlyeq 0
\quad 
\bmat{\mcl P  &\mcl{PC^*}\\ \mcl{CP}&\gamma I } \succcurlyeq 0
\end{equation}
\item  There exist $\mcl P=\mcl P^* \in \PI$, and $\gamma,\epsilon>0$, such that $\mcl P \succcurlyeq \epsilon I$, satisfying
\begin{equation}\label{eq:lpi6b}
\hspace{-6mm}\bmat{-\gamma I & \mcl{B^*PT}\\ \mcl{T P B} & \mcl{A^*PT+T^*PA}} \preccurlyeq0
\quad 
\bmat{\gamma I & \mcl C\\\mcl C^* & \mcl{T^*PT}}\succcurlyeq 0
\end{equation}

    \end{enumerate}

Then, $\norm{G(\mcl{T,A,B,C},0)}_{\infty,2} \leq \gamma$. 
\end{thm}
\begin{proof} We use the duality result of Thm.~\ref{thm:dualnorm}, to prove that the LPIs in statements $1)$ and $2)$ imply $\norm{G(\mcl{T,A,B,C},0)}_{\infty,2} \leq \gamma$. First, we show $1)$ implies $\norm{G(\mcl T',\mcl A',0,\mcl C',\mcl B')}_{2,i} \leq \gamma$.

Suppose $1)$ holds for some $\mcl P\succcurlyeq \epsilon I,\gamma>0$. Then, $\mcl P, \gamma$ satisfy Eqn.~\eqref{eq:lpi5b}.

Now, let $\mcl{T'=T^*}$, $\mcl{A'=A^*}$, $\mcl{B'=C^*}$, and $\mcl{C'=B^*}$. Then, statement $1)$ of Thm.~\ref{thm:LPIs_2,i_dilated} hold for $G(\mcl T',\mcl A',0,\mcl C',\mcl B')$:
\begin{align*}
\bmat{-\gamma I & \mcl C' \\(\mcl C')^* &(\mathcal A')^*\mathcal P\mathcal T'
+(\mathcal T')^*\mathcal P\mathcal A'}
\!\preccurlyeq\! 0, \;\bmat{\gamma I&(\mcl B')^*\mcl P\\ \mcl{P} \mcl B'&\mcl P} \!\succcurlyeq\! 0,
\end{align*}
which implies that \(\norm{G(\mcl T',\mcl A',0,\mcl C',\mcl B')}_{2,i}\leq \gamma.\)  

Next, we show that $2)$ implies $\norm{G(\mcl T',\mcl A',0,\mcl C',\mcl B')}_{2,i} \leq \gamma$. 
Suppose $2)$ holds for some $\mcl P\succcurlyeq \epsilon I,\gamma>0$. Then, $\mcl P, \gamma$ satisfy Eqn.~\eqref{eq:lpi6b}. Hence, statement $2)$ of Thm.~\ref{thm:LPIs_2,i_dilated} hold for $G(\mcl T',\mcl A',0,\mcl C',\mcl B')$. Specifically, we have that
{\small
\begin{align*}
       \bmat{-\gamma I &\hspace{-2 mm} \mcl{C'P(T')^*}\\ \mcl{T' P (C')^*} &\hspace{-2 mm} \mcl{A'P(T')^*+T'P(A')^*}} \preccurlyeq0,
         \bmat{\gamma I &\hspace{-1 mm} (\mcl B')^*\\\mcl B' &\hspace{-1 mm} \mcl{T'P(T')^*}}\succcurlyeq 0,
\end{align*}}
which implies that \(\norm{G(\mcl T',\mcl A',0,\mcl C',\mcl B')}_{2,i}\leq \gamma.\)  

Now, we are in position to invoke Thm.~\ref{thm:dualnorm}, using $G(\mcl T',\mcl A',0,\mcl C',\mcl B')$, which states that
\begin{align*}
    \norm{G(\mcl T',\mcl A',0,\mcl C',\mcl B')}_{2,i}&=\norm{G(\mcl T^*,\mcl A^*,0,\mcl B^*,\mcl C^*)}_{2,i}\\&\hspace{-9mm}=\norm{G(\mcl{(T')^*,(A')^*,(B')^*,(C')^*},0)}_{\infty,2}. 
\end{align*}

Since $\mcl{T,A,B,C} \in \PI \subset \mcl L(X)$, we have that, by Thm.~11.3 of Chapter II of~\cite{gohbergVol1}, $(\mcl T')^*=\mcl T^{**}=\mcl T$, $(\mcl A')^*=\mcl A^{**}=\mcl A$, $(\mcl B')^*=\mcl B^{**}=\mcl B$, and $(\mcl C')^*=\mcl C^{**}=\mcl C$, and hence 
\[
\norm{G(\mcl T',\mcl A',0,\mcl C',\mcl B')}_{2,i}=\norm{G(\mcl{T,A,B,C},0)}_{\infty,2}.
\]

Since both $1)$ and $2)$ imply $\norm{G(\mcl T',\mcl A',0,\mcl C',\mcl B')}_{2,i}\leq \gamma$, both imply $\norm{G(\mcl{T,A,B,C},0)}_{\infty,2}\leq \gamma$, as desired.
\end{proof}

\section{Numerical Results}\label{sec:numerics}

This section provides numerical validation of the theoretical results. This includes establishing the veracity of Eq.~\eqref{eq:intertwining2} 
as well as an application of all LPI results. 
In all cases, numerical simulation and LPI solving are achieved via the features offered by the PIETOOLS~\cite{manual2025} Matlab toolbox.

\subsection{Verifying Eqn.~\eqref{eq:intertwining2} via numerical simulations}

We verify the intertwining identity by numerically integrating $x$, and $v$, which satisfy the primal PIE Eqn.~\eqref{eq:PIE} and $\bar x$, $\bar v$, which satisfy the dual PIE Eqn.~\eqref{eq:dualPIE}. The numerical residual of the intertwining relationship, defined below as $\err$, is then an empirical metric for validation of the identity.\[
\err(T) = \bl(\mcl B^*\bar x(T) \br)^T v- \ip{\bar v}{\mcl C x}_{L_2[0,T]}.
\]
For the following two PDE systems, the initial condition $v=1$ (for 
primal) and input $\bar{v}(T-t) = \frac{1}{\sqrt{T}}$ (for the dual) are used. The resulting error, $\err$, for various time horizons $T$ are reported in Table \ref{tab:Intertwining}.\begin{itemize}
    \item[\bf E1.] $1$D Transport equation:
    {\small
    \begin{align}
    \hspace{- 6 mm} \partial_t x(t,s) \!=\! -\partial_sx(t,s) +sw(t), \,
    z(t) = \!\int_0^1 \!x(t,s) ds,\; x(t,0) =0 \label{eq:E1}
\end{align}}
\item[\bf E2.] $1$D Heat equation: 
\begin{align}
    &\partial_t x(t,s) = \partial_s^2x(t,s)-x(t,s)+(s-s^2)w(t)\nonumber\\
    &z(t) = \int_0^1 x(t,s) ds,\;x(t,1)=x(t,0) = 0\label{eq:E2}
\end{align}
\end{itemize}
\begin{table}[!h]
    \centering
    \begin{tabular}{|p{0.9cm}|p{1.35cm}|p{1.35cm}|p{1.35cm}|p{1.35cm}|}
    \hline
    Example&$err(0.25)$&$err(0.5)$&$err(0.75)$&$err(1)$\\\hline
    \textbf{E1} &$1\cdot 10^{-4}$&$7.1\cdot 10^{-5}$&$5.8\cdot 10^{-5}$&$5.0\cdot 10^{-5}$\\\hline
    \textbf{E2} &$3.3\cdot 10^{-5}$&$2.4\cdot 10^{-5}$&$1.9\cdot 10^{-5}$&$1.7\cdot 10^{-5}$\\\hline
        \end{tabular}
    \caption{Error for Numerically Validated Intertwining Relationship}
    \label{tab:Intertwining}
\end{table}

\subsection{Upper-Bounds on system norms and the $H_2$-norm}\label{subsec:upperbound}


\begin{table*}[!h]
    \centering
    \begin{tabular}{|c|c|c|c|c|c|c|c|c|}
    \hline
      Ex. \# & Thm. \ref{thm:LPI_2,i}\cellcolor{blue!15}  & Cor. \ref{cor:LPI_2,i_dual}\cellcolor{blue!15} &Thm. \ref{thm:LPIs_2,i_dilated} 1) \cellcolor{blue!15} & Thm. \ref{thm:LPIs_2,i_dilated}\cellcolor{blue!15}  2)&Thm. \ref{thm:LPI_inf,2}\cellcolor{red!15} & Cor. \ref{cor:LPI_inf,2_dual}\cellcolor{red!15}& Thm. \ref{thm:LPIs_inf,2_dilated} 1)\cellcolor{red!15} & Thm. \ref{thm:LPIs_inf,2_dilated} 2)\cellcolor{red!15}\\\hline
        1&3.6724 & N/A&3.6724 &N/A &3.6725&3.6724&3.6724 &3.6725\\\hline
        2&3.8300 &N/A &3.8300 &N/A &3.8300&3.8304&3.8304 &3.8300 \\\hline
        3  &0.2236&  0.2237&0.2236 &0.2237 &0.2242&0.2236&0.2237 &0.2242\\\hline
        4 &0.0353 &  0.0358&0.0353&0.0358 & 0.0356&0.0353&0.0353&0.0358 \\\hline
        5 & 1.1412 & 1.1442&1.1412&1.1441 &1.1544&1.1526&1.1526 & 1.1547\\\hline
        \end{tabular}
    \caption{Upper bounds on impulse-to-energy and energy-to-peak norms for various delay and PDE systems in Sec. \ref{subsec:upperbound}. The bounds are obtained by solving the LPIs presented in Secs. \ref{sec:LPIs} and \ref{sec:dilatedLPIs}. The color of column headers indicate whether the bounds are on the~\colorbox{blue!15}{$\norm{G}_{2,i}$} ~or~ the~\colorbox{red!15}{$\norm{G}_{\infty, 2}$} norm.}
    \label{tab:norms}
\end{table*}
Tab.~\ref{tab:norms} compares the results for both norms using the aforementioned theorems and corollaries. Examples 1 and 2 
were taken from~\cite{braghini2023}, as particular cases of the following delay system with N state and output delays $\tau_1<\dots<\tau_N$
\begin{align*}
				  \dot{x}(t) &= A_0  x(t) + \sum_{i=1}^NA_i  x(t-\tau_i) + B u(t), \; y(t)=C_0 x(t)
				\end{align*}

    Conversion of this system to PIEs is automated by the Matlab package PIETOOLS~\cite{manual2025}. For details on the PIE representation of delay systems, we refer to~\cite{peet2021}, whereas for the PIE representation of the PDEs in examples 3 through 5, we refer to~\cite{shivakumar2024}. The LPIs where tested by PIETOOLS with Lyapunov PI variables with polynomial kernels of degree $5$.

    The first important result is the verification of the symmetries of Thm.~\ref{thm:dualnorm} by comparing the first column (Thm.~\ref{thm:LPI_2,i}), to the second column (Cor.~\ref{cor:LPI_2,i_dual}) and checking that, when feasible, they produce the same value, apart from numerical truncation errors, for all the examples tested. For the DDE examples, Cor.~\ref{cor:LPI_2,i_dual}, and Thm.~\ref{thm:LPIs_2,i_dilated} 2) are infeasible. 
    Similarly, the fifth column (Thm.~\ref{thm:LPI_inf,2}) and sixth column (Cor.~\ref{cor:LPI_inf,2_dual}) produce the same result.
						
						
                
 \noindent \textbf{Example 1: }
				First consider the scalar, single delay system, where the non-zero matrices are $A_1=-1$, 
                $B=2$, and $C_0=2$, with the delay $\tau = 0.5$. This system has an analytic bound for the $H_2$ norm, which is $3.6724$. 
                As seen in Table~\ref{tab:norms}, bounds obtained by LPIs closely match the analytical bound.
			
			 \noindent \textbf{Example 2: }
				Next, consider the two-delay system with $\tau_1 = \pi/10$, $\tau_2 = 1$, so that the delays are incommensurate, and
				\begin{align*}
					&A_0 = \bmat{
						-1 & 1 & 2\\
						1 & -3 & 2\\
						0 & 0 & -1
					},\qquad  A_1 = \frac{1}{5}\bmat{
						-3 & 0 & 1\\
						1 & -2 & 0\\
						0 & 2 & -2
					},\\
					& A_2 =  \frac{1}{5}\bmat{
						-4 & 1 & 0\\
						0 & -2 & 1\\
						2 & 1 & -3
					}, \quad B = \bmat{	1 \\ 1 \\1}, \; C_0 = \bmat{1 & 1 &1 }.
				\end{align*}
				The method in~\cite{jarlebring2011} and the bounds from the LPIs agree to four significant figures. However, 10th order Pad\'e approximation yields a 
                higher value of of $3.9275$ for the $H_2$--norm. 
\noindent \textbf{Example 3: Pure Convection.}
This example again considers Eqns. \eqref{eq:E1}. 
An analytical bound is obtained by choosing $v=1$, and setting the initial condition on the PDE state, $\mcl{T}x(0) = \mcl Bv = \mcl B$. Then, from standard PDE theory, analytical solutions for the state as a function of time are known. Finally, the $L_2$ norm of this solution is calculated to obtain $\norm{G}_{2,i} = 0.2236$. For numerical bounds, all LPI tests and their corresponding results are shown in Tab. \ref{tab:norms}. The analytical bound and those shown in Tab. \ref{tab:norms} agree.

\noindent \textbf{Example 4: Reaction-Diffusion.}
Next, we consider the reaction-diffusion PDE described in Eqns. \eqref{eq:E2}.
LPI results for this example are again presented in Tab. \ref{tab:norms}. An analytical bound for the $H_2$--norm, derived from eigenfunction expansion of the solution, is $0.0353$. All the bounds obtained by solving the LPIs, proposed in the paper, agree with the analytical bound and are accurate up to 3 decimal places.

\noindent \textbf{Example 5: Timoshenko Beam Equations.}
Lastly, we consider a beam with transverse displacement $u(t,s)$ and cross-section rotation angle $\phi(t,s)$. Considering two disturbance inputs $w_1(t)$ and $w_2(t)$, The Timoshenko beam model relates these quantities by the following equations.
\begin{align}
&\ddot{u} = \partial_s(u_s - \phi)-0.1\dot{u}+sw_1, \;\;  \nonumber\\&\ddot{\phi} = \phi_{ss} -0.1\dot{\phi}+(u_s - \phi)+sw_2,  \nonumber 
\\ &\phi(t,0) = u(t,0) = 0, \; \phi_s(t,1) = 0, \; u_s(t,1) - \phi(t,1) = 0. \nonumber
\end{align}

These equations have the first-order in time, hyperbolic representation with $\mbf v = \bmat{
    \dot{u},  u_s -\phi, \dot{\phi}, \phi_s}$,

\begin{align*}
    &\dot{\mbf v}(t,s) \!\!=\!\! -\!\begin{bmatrix}
        0.1 \!&\! 0\!&\!0\!&\!0\\ 0\!&\!0\!&\!1\!&\!0\\ 0 \!&\!-1\! &\!0.1\!&\!0\\ 0\!&\!0&\!0\!&\!0
    \end{bmatrix} \mathbf{v}(t,s) \!+\! \begin{bmatrix}
        0&\!1&\!0&\!0\\ 1&\!0&\!0&\!0\\ 0&\!0&\!0&\!1\\ 0&\!0&\!1&\!0
    \end{bmatrix} \partial_s \mbf{v}(t,s) \\&\hspace{2cm}+ \begin{bmatrix}
        s&0&0&0\\0&0&s&0
    \end{bmatrix}^T \bmat{w_1(t)\\w_2(t)}\nonumber,    
\end{align*} with boundary conditions 
$\mathbf{v}_1(t,0) = \mathbf{v}_3(t,0) = \mathbf{v}_4(t,1) = \mathbf{v}_2(t,1) = 0 $. We consider the two outputs
\[
\bmat{z_1(t)\\z_2(t)} = \int_0^1 \bmat{
        1 & 0&0&0\\ 0 &0 &1&0
    } \mathbf{v}(t,s) ds.\]
The results of Table~\ref{tab:norms} show that, as expected for multiple-input-multiple-output systems, the two proxy norms do not match, although both are lower bounds to the $H_2$--norm, from Prop.~\ref{prop:H2equi_PIE}. Furthermore, Prop.~\ref{prop:H2equi_PIE} states that the $H_2$ norm is upper bounded by $\sqrt{n_w}\|G\|_{2,i}=\sqrt{2}\|G\|_{2,i}$ and $\sqrt{n_z}\|G\|_{\infty,2}=\sqrt{2}\|G\|_{2,i}$. The upper-bounds of Table~\ref{tab:norms} for the proxy norms give an upper-bound of $\sqrt{2}\|G\|_{2,i}\approx1.6139$. Applying finite element discretization to the Timoshenko beam equation, with $160$ elements, the $H_2$--norm was computed as $1.6077$.

\section{Conclusions}\label{sec:conclusion}
\vspace{-1mm}

Primal and dual characterization of the $H_2$-norm is essential for controller synthesis. For systems without well-defined transfer functions, however, such characterizations are not available. In this paper, a duality is established between two proxies for $H_2$ norm -- the impulse-to-energy (A) and the energy-to-peak norms (B). This duality is established for the class of linear Partial Integral Equations used to represent linear PDEs and time-delay systems. Specifically, for each norm, we propose a map from primal system (with norm A/B) to dual system (with norm B/A) and show that in both cases the A/B norm of the primal is the B/A norm of the dual. This is established through an intertwining theorem relating primal and reverse-time dual solutions. Using this duality, for each norm, both primal and dual convex operator inequalities are proposed which can be tested using convex optimization. Both the intertwining relation and LPI-obtained norm bounds are verified numerically on linear PDE and time-delay systems.\vspace{-5mm}

\bibliographystyle{IEEEtran}

\bibliography{ACC_27_duality}

\end{document}